\documentclass{amsart}

\usepackage{amsmath}
\usepackage{amsfonts}
\usepackage{amssymb}
\usepackage{graphicx}
\usepackage{mathrsfs}
\usepackage{pb-diagram}
\usepackage{epstopdf}

\newtheorem{theorem}{Theorem}[section]

\newtheorem{corollary}[theorem]{Corollary}
\newtheorem{prop}[theorem]{Proposition}
\newtheorem{definition}[theorem]{Definition}

\newtheorem{remark}[theorem]{Remark}

\numberwithin{equation}{section}

\newcommand{\conn}{\nabla}

\newcommand{\der}{\mathrm{d}}

\newcommand{\pairing}[2]{\left( #1\, , \, #2 \right)}

\newcommand{\integer}{\mathbb{Z}}
\newcommand{\rat}{\mathbb{Q}}
\newcommand{\real}{\mathbb{R}}
\newcommand{\cpx}{\mathbb{C}}

\newcommand{\consti}{\mathbf{i}\,}

\newcommand{\sphere}[1]{\mathbf{S}^{#1}}
\newcommand{\proj}{\mathbb{P}}

\newcommand{\SU}{\mathrm{SU}}
\newcommand{\Sp}{\mathrm{Sp}}

\begin{document}

\title[Mirror maps equal SYZ maps for toric CY surfaces]{Mirror maps equal SYZ maps\\for toric Calabi-Yau surfaces}

\author[S.-C. Lau]{Siu-Cheong Lau}
\address{Institute for the Physics and Mathematics of the Universe (IPMU) \\ University of Tokyo \\ Kashiwa \\ Chiba 277-8583 \\ Japan}
\email{siucheong.lau@ipmu.jp}
\author[N.C. Leung]{Naichung Conan Leung}
\address{The Institute of Mathematical Sciences and Department of Mathematics\\ The Chinese University of Hong Kong\\ Shatin \\ Hong Kong}
\email{leung@math.cuhk.edu.hk}
\author[B.S. Wu]{Baosen Wu}
\address{Harvard University\\
   Cambridge, MA 02138\\
   USA}
\email{baosenwu@gmail.com}

\begin{abstract}
We prove that the mirror map is the SYZ map for every toric Calabi-Yau surface.  As a consequence one obtains an enumerative meaning of the mirror map.  This involves computing genus-zero \emph{open} Gromov-Witten invariants, which is done by relating them with closed Gromov-Witten invariants via compactification and using an earlier computation by Bryan-Leung.
\end{abstract}

\maketitle

\section{Introduction}

Mirror map has been an essential ingredient in the study of mirror symmetry for Calabi-Yau manifolds.  It gives a canonical local isomorphism between the K\"ahler moduli and the mirror complex moduli near the large complex structure limit.  Enumerative predictions can only be made in the presence of mirror map, so that one can identify Yukawa couplings among the mirror pair.

Yet geometric meanings of the mirror map remain unclear to mathematicians.  Integrality of coefficients of certain expansion of the mirror map have been studied (see, for example, \cite{lian98,zudilin02,krattenthaler10}), and it is expected that these coefficients contain enumerative meanings.  This paper obtains such a meaning in the study of mirror symmetry for toric (non-compact) Calabi-Yau surfaces.

Let $X$ be a toric Calabi-Yau $n$-fold.  Hori-Vafa \cite{hori00} has written down the mirror family of $X$ as hypersurfaces in $\cpx^2 \times (\cpx^\times)^{n-1}$ via physical considerations.  On the other hand, Strominger-Yau-Zaslow \cite{syz96} proposed a general principle that the mirror should be constructed via T-duality, which is, roughly speaking, taking dual torus fibrations.  From this SYZ perspective a natural question arises: can the mirror written down by Hori-Vafa be obtained by T-duality?

This question has an affirmative answer \cite{CLL}: By taking dual torus bundles and Fourier transform of open Gromov-Witten invariants of $X$ which admits wall-crossing in the sense of Auroux \cite{auroux07,auroux09}, the mirror $\check{X}$ (as a complex manifold) was written down (in Theorem 4.38 of \cite{CLL}) explicitly in terms of K\"ahler parameters and open Gromov-Witten invariants of $X$ (as a symplectic manifold), and this result agrees with Hori-Vafa's one in the sense that $\check{X}$ appears as a member of Hori-Vafa's mirror family.

While the Hori-Vafa recipe gives the mirror complex moduli, this SYZ approach gives an explicit map, which we call the SYZ map, from the K\"ahler moduli to the mirror complex moduli.  An immediate question is, does it agree with the mirror map, which pulls back the mirror canonical complex coordinates to the canonical K\"ahler coordinates? The paper \cite{CLL} has studied examples such as $K_{\proj^1}$ and $K_{\proj^2}$, and the SYZ maps coincide with the mirror maps in these examples.

This paper gives an affirmative answer to this question for $n=2$.  The main result is Theorem \ref{can_coord}, and for convenience we restate it here in one sentence:

\begin{theorem} [(Restatement of Theorem \ref{can_coord})] \label{main_theorem}
For every toric Calabi-Yau surface, the mirror map is the SYZ map.
\end{theorem}

Now since the SYZ map is written down in terms of enumerative invariants (namely, the one-pointed genus-zero open Gromov-Witten invariants of $X$), we obtain a geometric meaning of the mirror map.  Moreover, in these cases the open Gromov-Witten invariants are indeed integer-valued.  As a result, one obtains integrality of the coefficients of the mirror map.

To prove this theorem, we need to compute one-pointed genus-zero open Gromov-Witten invariants of a Lagrangian toric fiber $\mathbf{T} \subset X$.  Our strategy is to relate the open invariants to some closed invariants of $\bar{X}$, where $\bar{X}$ is a suitable toric compactification of $X$.  Then by the results of Bryan-Leung \cite{bryan-leung00} when they compute the Yau-Zaslow numbers for elliptic K3 surfaces, we obtain the answers for these open invariants (see Theorem \ref{Thm_openGW}).  This strategy is based on a generalization of the relation between open and closed invariants proved by Chan \cite{Chan10}, and this strategy has also been used in \cite{LLW10} for computing open invariants of certain toric Calabi-Yau threefolds.

Since we are in the $\dim_{\cpx} = 2$ situation so that every Calabi-Yau is automatically hyper-K\"ahler, there is another approach to mirror symmetry via hyper-K\"ahler twist.  We'll see (in Section \ref{hyperKaehler}) that the SYZ mirror is consistent with this hyper-K\"ahler perspective.

The organization of this paper is as follows.  A short review on toric manifolds (with an emphasis on its symplectic geometry) is given in Section \ref{toric}.  Then in Section \ref{SYZ} we specialize the SYZ mirror construction proposed in \cite{CLL} to toric Calabi-Yau surfaces.  Section \ref{main_section} is the main section, which computes the open Gromov-Witten invariants and proves Theorem \ref{main_theorem}.

\begin{remark}
Having computed the open invariants, we see that the mirror $\check{X}$ constructed via the SYZ approach agrees with the one written down by Hosono, who approached the subject from the perspective of hypergeometric series instead.
\end{remark}

\section*{Acknowledgements}
We are grateful to Kwokwai Chan for helpful discussions on his work on mirror symmetry for toric nef manifolds \cite{Chan10} and bringing our attention to the work of Hosono \cite{hosono06}.  The first author would like to thank Cheol-Hyun Cho for enlightening discussions on open Gromov-Witten invariants at Seoul National University on June, 2010.   The first and second author would like to thank Andrei C\u ald\u araru and Yong-Geun Oh for the hospitality and joyful discussions at University of Wisconsin, Madison.

The work of the first author was supported by World Premier International Research Center Initiative (WPI Initiative), MEXT, Japan.  The work of the second author described in this paper was substantially supported by a grant from the Research Grants Council of the Hong Kong Special Administrative Region, China (Project No. CUHK401809).

\section{Toric Calabi-Yau surfaces} \label{toric}
 
\subsection{A quick review on toric manifolds}
Let's begin with some notations and terminologies for toric manifolds.  Let $N \cong \integer^n$ be a lattice, and for simplicity we'll always use the notation $N_R := N \otimes R$ for a $\integer$-module $R$.  From a simplicial convex fan $\Sigma$ supported in $N_\real$ we obtain a toric complete complex $n$-fold $X = X_\Sigma$ which admits an action from the complex torus $N_\cpx / N \cong (\cpx^\times)^n$, which accounts for its name `toric manifold'.  There is an open orbit in $X_\Sigma$ on which $N_\cpx / N$ acts freely, and by abuse of notation we'll also denote this orbit by $N_\cpx / N \subset X_\Sigma$.

We denote by $M$ the dual lattice of $N$.  Every lattice point $\nu \in M$ gives a nowhere-zero holomorphic function
$\exp 2\pi\consti \pairing{\nu}{\cdot} : N_\cpx / N \to \cpx $
which extends as a meromorphic function on $X_\Sigma$.  Its zeroes and poles give a toric divisor which is linearly equivalent to $0$.  (A divisor $D$ in $X_\Sigma$ is toric if $D$ is invariant under the action of $N_\cpx / N$ on $X_\Sigma$.)

If we further equip $X_\Sigma$ with a toric K\"ahler form $\omega$, then the action of $\mathbf{T} := N_\real / N$ on $X_\Sigma$ induces a moment map
$$\mu_0: \proj_{\Sigma} \to M_\real$$
whose image is a polyhedral set $P \subset M_\real$ defined by a system of inequalities
$$\pairing{v_j}{\cdot} \geq c_j$$
where  $v_j \in N$ for $j=0, \ldots, m$ are all primitive generators of rays of $\Sigma$, and $c_j \in \real$ are some suitable constants.

The polyhedral set $P$ admits a natural stratification by its faces.  Each codimension-one face $T_j \subset P$ which is normal to $v_j \in N$ corresponds to an irreducible toric divisor $D_j = \mu_0^{-1} (T_j) \subset X_\Sigma$ for $j = 0, \ldots, m$, and all other toric divisors are generated by $\{D_j\}_{j=0}^{m}$.  For example, the anti-canonical divisor $K_X^{-1}$ is $\sum_{j=0}^{m} D_j$.

\subsection{Classification of toric Calabi-Yau surfaces}

\begin{definition}
A toric manifold $X = X_\Sigma$ is Calabi-Yau if its anti-canonical divisor $K^{-1}_X = \sum_{i=0}^{m} D_i$ is linearly equivalent to $0$ in a toric way, in the sense that there exists an $N_\cpx / N$-invariant holomorphic function whose zero divisor is $K^{-1}_X$.
\end{definition}

We notice that by definition a toric Calabi-Yau possesses a non-zero holomorphic function, and hence it must be non-compact.  Since every $N_\cpx / N$-invariant holomorphic function is of the form $\exp 2\pi\consti (\underline{\nu}, \cdot)$ for some $\underline{\nu} \in M$, an alternative definition is that there exists $\underline{\nu} \in M$ such that $\pairing{\underline{\nu}}{v_i} = 1$ for all primitive generators $v_i \in N$ of rays of $\Sigma$.

A toric Calabi-Yau manifold possesses a holomorphic volume form, which is locally written as $\der\zeta_0 \wedge \ldots \wedge \der\zeta_{n-1}$, where $\{\zeta_i\}_{i=0}^{n-1}$ are local complex coordinates corresponding to the basis dual to $\{v_i\}_{i=0}^{n-1}$. In this paper we'll concentrate on toric Calabi-Yau surfaces, which is classified by the number of rays in its fan:

\begin{prop}
Let $\Sigma_m$ be the convex fan supported in $\real^2$ whose rays are generated by $(i,1)$ for $i = 0, \ldots, m$.  Then $X_{\Sigma_m}$ is a toric Calabi-Yau surface.  Conversely,  if $X_{\Sigma}$ is a toric Calabi-Yau manifold, then $X_{\Sigma} \cong X_{\Sigma_m}$ as toric manifolds for some $m \geq -1$.  ($m = -1$ means that the fan is $\{0\}$ and so $X_{\Sigma_m} \cong (\cpx^\times)^2$.
\end{prop}

\begin{proof}
Taking $\underline{\nu} = (0,1) \in \integer^2$, one has $\pairing{\underline{\nu}}{(i,1)} = 1$ for all $i = 0, \ldots, m$.  Thus $X_{\Sigma_m}$ is a toric Calabi-Yau surface.

Now suppose $X_{\Sigma}$ is a toric Calabi-Yau surface whose fan $\Sigma$ has rays generated by $v_i \in N$ for $i = 0, \ldots, m$.  We may take $\{v_0, v_1\}$ as a basis of $N$ and identify it with $\{(0,1), (1,1)\} \subset \integer^2$.  Then $\pairing{\underline{\nu}}{v_0} = \pairing{\underline{\nu}}{v_1} = 1$ implies that $\underline{\nu}$ is identified with $(0,1)$.  Moreover, since for each $i = 0, \ldots, m$, $\pairing{\underline{\nu}}{v_i} = 1$, $v_i$ must be identified with $(k_i, 1)$ for some $k_i \in \integer$.  Without lose of generality we may assume that $v_0, \ldots, v_m$ are labeled in the clockwise fashion, so that $\{k_i\}$ is an increasing sequence.  Inductively, using the fact that $\{v_{i-1}, v_{i}\}$ is simplicial, one can see that $k_i = i$ for all $i = 0, \ldots, m$.
\end{proof}

\begin{remark}
Every toric Calabi-Yau surface $X_{\Sigma_m}$ for $m \geq 1$ is the toric resolution of $A_{m-1}$ singularity $\cpx^2 / \integer_{m}$, whose fan is the cone $\real_{\geq 0}\langle (0,1), (m,1) \rangle \subset \real^2$.  (See Figure \ref{A_n resolution}.) $\{D_i\}_{i=1}^{m-1}$ is the set of compact irreducible toric divisors, and it generates $H_2(X, \integer)$.  The K\"ahler moduli of $X_{\Sigma_m}$ has canonical K\"ahler coordinates given by
$$q_i := \exp \left(- \int_{D_i} \omega\right)$$
for $i = 1, \ldots, m-1$.
\end{remark}

\begin{figure}[htp]
\caption{Toric resolution of $\cpx^2 / \integer_{m}$.} \label{A_n resolution}
\begin{center}
\includegraphics[height=208pt,width=324pt]{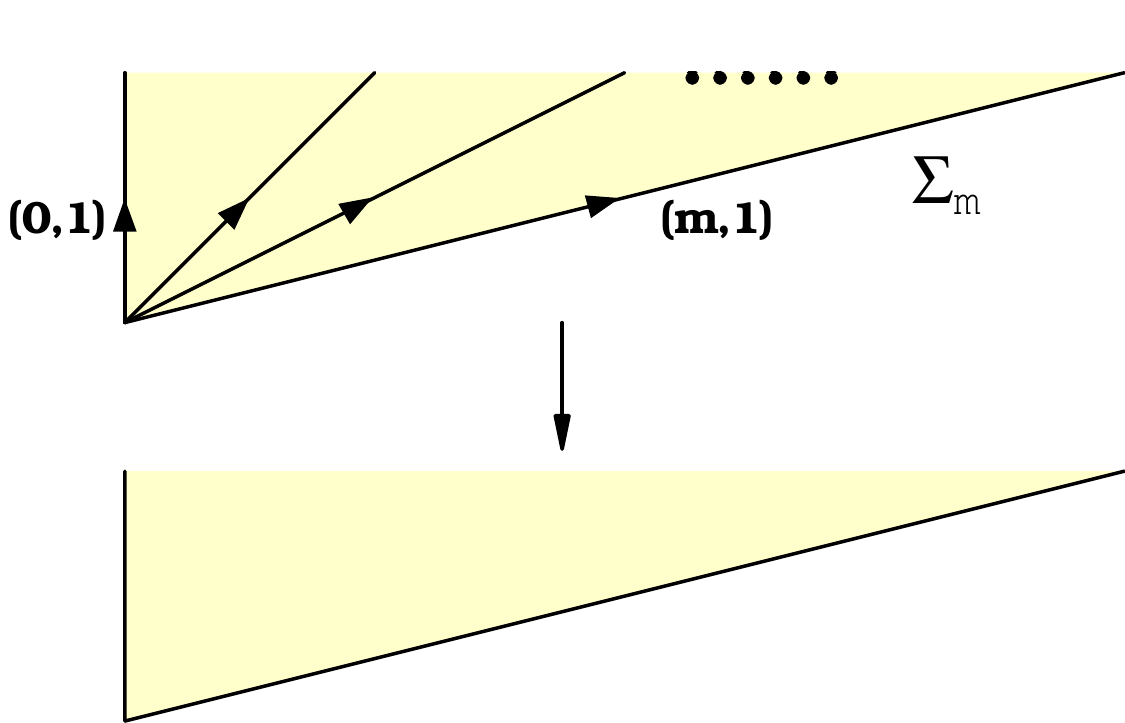}
\end{center}
\end{figure}

\subsection{Symplectic invariants}
We would be interested in the symplectic geometry of $X = X_{\Sigma_m}$.  This subsection gives a brief review on some important symplectic invariants that we'll use later.

For a Lagrangian torus $T$ in a symplectic manifold $(X, \omega)$, let $\pi_2(X,T)$ denote the group of homotopy classes of maps
$$u: (\Delta, \partial\Delta) \to (X, T)$$
where $\Delta := \{z \in \cpx: |z| \leq 1\}$ denotes the closed unit disk in $\cpx$.  For $\beta \in \pi_2(X,T)$, the two most important classical symplectic invariants are its symplectic area $\int_\beta \omega$ and its Maslov index $\mu (\beta)$.  Moreover, we have the open Gromov-Witten invariants defined by FOOO \cite{FOOO_I,FOOO_II} which is central to the study of mirror symmetry:

\begin{definition} [(\cite{FOOO_I,FOOO_II})]
Let $X$ be a symplectic manifold together with a choice of compatible almost complex structure.  Given a Lagrangian torus $T \subset X$ and $\beta \in \pi_2 (X, T)$, the genus zero one-pointed open GW-invariant $n^T_\beta$ is defined as
$$n^T_\beta := \pairing{[\mathcal{M}_1(T, \beta)]}{[\mathrm{pt}]}.$$
In the above expression $\mathcal{M}_1(T, \beta)$ is the moduli space of stable maps $(\Sigma, \partial\Sigma, p_0) \to (X,T)$ where $\Sigma$ is a genus zero Riemann surface with a connected boundary $\partial\Sigma$ and $p_0 \in \partial\Sigma$.  $[\mathcal{M}_1(T, \beta)] \in H_n(T, \rat)$ denotes its virtual fundamental chain, so that we may take the Poincar\'e pairing with the point class $[\mathrm{pt}] \in H_0(T, \integer)$ to give a rational number.
\end{definition}

From now on we may write $n_\beta = n^T_\beta$.  Recall that the moduli space $\mathcal{M}_k(T, \beta)$ of stable disks with $k$ marked points representing $\beta$ has expected dimension $n + \mu(\beta) + k - 3$.  In our situation $k = 1$, and so the expected dimension is $n + \mu(\beta) - 2$, which matches with $\dim T = n$ if and only if $\mu(\beta) = 2$.  Thus $n_\beta \not= 0$ only when $\mu(\beta) = 2$.  Coming back to toric manifolds, we have the following result by Cho-Oh \cite{cho06} and FOOO \cite{FOOO1}:

\begin{prop} [(\cite{cho06,FOOO1})] \label{Cho-Oh}
Let $X$ be a toric manifold and $\mathbf{T} \subset X$ be a Lagrangian toric fiber.  One has $n^{\mathbf{T}}_{\beta_i} = 1$ where $\beta_i \in \pi_2 (X, \mathbf{T})$ are the basic disk classes which are of Maslov index two.  Moreover, for all $\beta \in \pi_2 (X, \mathbf{T})$, $n_\beta \not= 0$ only when $\beta = \beta_i + \alpha$ for some $i = 1, \ldots, m-1$ and $\alpha \in H_2(X)$ represented by some rational curves with $K_X \cdot \alpha = 0$.
\end{prop}

In the above proposition, $n_\beta$ is explicitly known in unobstructed situations.  When $X$ is non-Fano and $\beta = \beta_i + \alpha$ for $\alpha \in H_2(X) - \{0\}$, $\mathcal{M}_1(\mathbf{T}, \beta)$ may be obstructed which makes it difficult to compute $n_\beta$.  In Section \ref{open_GW} we'll overcome this problem when $X$ is a toric Calabi-Yau surface.

\section{The mirror of a toric Calabi-Yau surface via SYZ} \label{SYZ}
Via SYZ construction, the mirror of a toric Calabi-Yau manifold $X$ of any dimension is written down in terms of K\"ahler parameters and open Gromov-Witten invariants of $X$ \cite{CLL}.  Restricting to $\dim X = 2$, the result is:

\begin{theorem}[(Surface case of Theorem 4.38 in \cite{CLL})] \label{mirror theorem}
Let $X = X_{\Sigma_m}$ be a toric Calabi-Yau surface.  By SYZ construction the mirror of $(X, \omega)$ is the complex manifold
$$\check{X} := \left\{(z,u,v) \in \cpx^\times \times \cpx^2 : uv = 1 + \sum_{i=1}^{m} \left( \prod_{j=1}^{i-1} q_j^{i-j} \right) (1+\delta_i) z^i \right\}$$
where
$$q_j := \exp \left(- \int_{D_j} \omega \right) \textrm{ for } j = 1, \ldots, m-1$$
are parameters recording symplectic areas of the compact irreducible toric divisors $D_1,\cdots,D_{m-1} \subset X$, and
$$ \delta_i := \sum_{\alpha \not= 0} n^{\mathbf{T}}_{\beta_i + \alpha} \exp\left(- \int_\alpha \omega \right) \textrm{ for } i = 1, \ldots, m-1$$
are `correction' terms in which the summation is over all $\alpha \in H_2 (X, \integer) - \{0\}$ represented by rational curves, $n^{\mathbf{T}}_{\beta_i + \alpha}$ are the open Gromov-Witten invariants of a Lagrangian toric fiber $\mathbf{T} \subset X$ for the disk classes $\beta_i + \alpha \in \pi_2 (X, \mathbf{T})$ , and $\delta_{m}$ is $0$.
\end{theorem}

\begin{remark}
After the open Gromov-Witten invariants $n^{\mathbf{T}}_{\beta_i + \alpha}$ are computed explicitly, we'll see (in Corollary \ref{def_eq}) that the defining equation of $\check{X}$ is simply
$$ uv = (1+z)(1+q_1 z)(1+q_1 q_2 z) \ldots (1 + q_1 \ldots q_{m-1} z). $$
Thus the mirror $\check{X}$ is a smoothing of the $A_{m-1}$ singularity $\cpx^2 / \integer_{m}$.  As $X$ degenerates to $\cpx^2 / \integer_{m}$, $q_j \to 1$ for all $j = 1, \ldots, m-1$, and so the mirror $\check{X}$ deforms to
$$\cpx^2 / \integer_{m} \cong \{ uv = (1+z)^{m} \}.$$
This class of mirror manifolds has already been investigated by Hosono \cite{hosono06} from the physical point of view, and we arrive at the same conclusion from the SYZ construction.
\end{remark}

In this section we give a very brief description to the SYZ mirror construction specialized to two-dimensional toric Calabi-Yaus.  The readers are referred to \cite{CLL} for details in all dimensions.

\subsection{T-duality} \label{T-duality}

The SYZ approach \cite{syz96} proposed that mirror symmetry is done by taking dual torus fibrations.  To do this we need a Lagrangian torus fibration over $X = X_{\Sigma_m}$, and this has been written down by Gross \cite{gross_examples}:
$$ \mu = ([\mu_0], |w - K| - K): X \to \frac{\real^2}{\real\langle (0,1) \rangle} \times \real \cong \real^2 $$
where $K \in \real_+$, $w$ is a holomorphic function on $X$ locally written as $\zeta_1 \zeta_2$ on each toric affine coordinate patch $\mathrm{Spec}(\cpx[\zeta_1, \zeta_2])$, and $\mu_0: X \to P \subset \real^2$ is the moment map.

The image of $\mu$ is the closed upper half plane $B = \real \times \real_{\geq -K}$.  The discriminant loci of $\mu$ consist of $\partial B = \real \times \{-K\}$ and isolated points $Q_i = ([T_{i-1,i}], 0) \in B$ for $i = 1, \ldots, m$, where each $T_{i-1,i}$ is a vertex of $P$ adjacent to the edges $T_{i-1}$ and $T_{i}$. (See Figure \ref{A_n base}.)

\begin{figure}[htp]
\caption{The base of $\mu$.} \label{A_n base}
\begin{center}
\includegraphics[height=140pt,width=340pt]{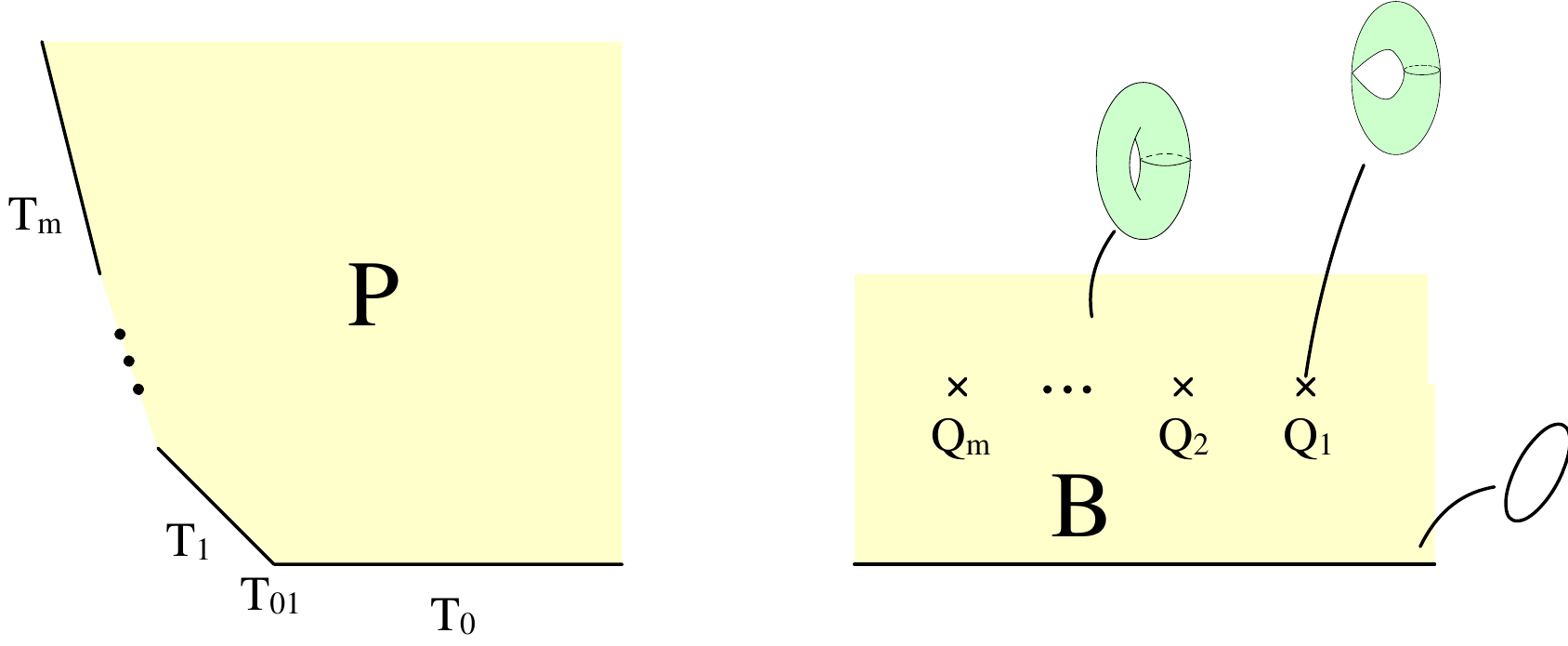}
\end{center}
\end{figure}

Let
$$ B_0 = \real \times \real_{> -K} - \{ Q_1, \ldots, Q_{m} \} $$
be the complement of discriminant loci in $B$.  The fiber of $\mu$ at $r \in B_0$ is denoted as $F_r$.  Away from the discriminant loci one may take the dual torus bundle:
$$\check{\mu}: \check{X}_0 := \big\{ (F_r, \conn): r \in B_0, \conn \textrm{ is a flat $U(1)$-connection on $F_r$} \big\} \to B_0$$
which is referred as the semi-flat mirror \cite{boss01}.  $\check{X}_0$ has semi-flat complex coordinates $(z_1, z_2)$: Let the coordinates of $Q_1$ be $(a,0)$ and 
$$U = B_0 - \{(r_1, 0) \in B_0: r_1 \leq a \}$$
which is a contractible open set in $B_0$ as shown in Figure \ref{A_n_disks}, and  $\lambda_i \in \pi_1(F_r)$ ($i = 1, 2$) are represented by the boundaries of the two disks $\Delta_i$ as shown in the diagram.  Then for $(F_r, \conn) \in \check{\mu}^{-1} (U)$,
$$ z_i (F_r, \conn) := \exp \left( -\int_{\Delta_i (r)} \omega \right) \mathrm{Hol}_{\conn} (\lambda_i).$$

\begin{figure}[htp]
\caption{The disks $\Delta_1$ and $\Delta_2$.} \label{A_n_disks}
\begin{center}
\includegraphics[height=93pt, width=199pt]{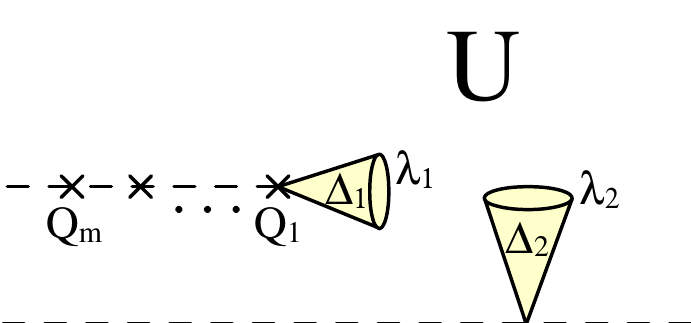}
\end{center}
\end{figure}

The above construction of semi-flat mirror complex manifolds has been discussed in a lot of literatures such as \cite{boss01}, and it is proposed that the semi-flat complex structure has to be corrected for compactifications \cite{gross07}.  The following section gives a brief review on these quantum corrections, which have been carried out in detail for general toric Calabi-Yaus in \cite{CLL}.

\subsection{Wall-crossing and the mirror complex coordinates} \label{mir_cpx_coord}

An essential ingredient of quantum corrections is the open Gromov-Witten invariant $n_\beta^{F_r}$, which exhibits the wall-crossing phenomenon in the sense of Auroux \cite{auroux07,auroux09} (Various examples such as $\cpx^2$, $\cpx^3$ and the Hirzebruch surface $\mathrm{F}_2$ have been discussed by Auroux to explain this wall-crossing phenomenon):

\begin{prop}[(see Section 4.5 of \cite{CLL} for the precise statement and proof)]

Let $X = X_{\Sigma_m}$ be a toric Calabi-Yau surface, and $H := \real \times \{0\} \subset B$ which is referred as `the wall'.  Write $B_0 - H = B_+ \cup B_-$.

\noindent For $r \in B_+$, $$n^{F_r}_\beta = n^{\mathbf{T}}_\beta$$ for all $\beta \in \pi_2(X,F_r)$, where $\mathbf{T} \subset X$ is a Lagrangian toric fiber.  

\noindent On the other hand, for $r \in B_-$, $n^{F_r}_\beta = 0$ for all $\beta$ except only one class $\beta_0$, and $n^{F_r}_{\beta_0} = 1$.

\end{prop}

The term `wall-crossing' refers to the phenomenon that $n^{F_r}_\beta$ jumps as $r$ crosses the wall $H$.  As a consequence, the superpotential $W$, which is a function on the semi-flat mirror $\check{X}_0$ defined by
$$ W(F_r, \conn) := \sum_{\beta \in \pi_2(X,F_r)} n^{F_r}_\beta \exp\left(-\int_\beta \omega\right) \mathrm{Hol}_\conn (\partial \beta), $$
also jumps when $r$ crosses the wall $H$.  To remedy this, the crucial idea is \emph{to use $W$ and $z_1$ as the mirror coordinate functions.} (In general Fourier transform of generating functions counting stable disks emanating from boundary divisors should be used as the mirror coordinates.)  After some computations (see \cite{CLL} for details) one sees that the mirror is of the form
$$\check{X} := \left\{(z,u,v) \in \cpx^\times \times \cpx^2: uv = g(z) \right\}$$
which is glued by two semi-flat pieces $\check{X}_+ = \check{X}_- = \cpx^\times \times \cpx$ (which contain $\check{\mu}^{-1}(B_\pm)$ respectively), where the coordinate charts are given by
$\iota_+: \check{X}_+ \to \check{X}$,
\begin{equation} \label{iota+}
\iota_+ (z_1, z_2) =  (z_1, z_2 g(z_1) ,z_2^{-1})
\end{equation}
and $\iota_-: \check{X}_- \to \check{X}$,
\begin{equation} \label{iota-}
\iota_- (z_1, z_2) =  (z_1, z_2 ,z_2^{-1} g (z_1)).
\end{equation}
More explicitly,
$$ g(z) := 1 + \sum_{i=1}^{m} \left( \prod_{j=1}^{i-1} q_j^{i-j} \right) (1+\delta_i) z^i $$ 
is the `gluing function' in the sense of Gross and Siebert \cite{gross07}, where
$$q_j := \exp \left(- \int_{D_j} \omega \right) \textrm{ for } j = 1, \ldots, m-1$$
and
$$ \delta_i := \sum_{\alpha \not= 0} n^{\mathbf{T}}_{\beta_i + \alpha} \exp\left(- \int_\alpha \omega \right) \textrm{ for } i = 1, \ldots, m.$$

With this correction the superpotential $W$, which takes values $z_2 g(z_1)$ on $\check{X}_+$ and $z_2$ on $\check{X}_-$, glues up to give the holomorphic function $u$ on the mirror $\check{X}$.

We see that in order to write down $\check{X}$ explicitly, one needs to compute the open Gromov-Witten invariants $n^{\mathbf{T}}_{\beta_i + \alpha}$, and this will be done in Section \ref{open_GW}.

\section{The mirror map is SYZ map} \label{main_section}
By the SYZ construction explained in the last section, each toric Calabi-Yau surface $(X, \omega)$ is associated with a complex surface $\check{X}$.  We call this to be the SYZ map which is a map from the K\"ahler moduli of $X$ to the complex moduli of $\check{X}$.  Now comes a crucial question: Does the SYZ map give the mirror map (Conjecture 5.1 of \cite{CLL})?

The mirror map is a local isomorphism between the K\"ahler moduli of $X$ and the complex moduli of $\check{X}$ such that it pulls back canonical coordinates on the complex moduli to canonical K\"ahler coordinates on the K\"ahler moduli.  (We recall that canonical K\"ahler coordinates are given by the symplectic areas of two-cycles in $X$, and canonical complex coordinates are given by the periods of $\check{X}$.)  In Hori-Vafa recipe, the mirror family is
$$\check{X}_{C_0, \ldots, C_{m}} = \left\{(z,u,v) \in \cpx^\times \times \cpx^2 : uv = \sum_{i=0}^{m} C_i z^i \right\}$$
where $C_i \in \cpx$ for $i = 0, \ldots, m$.  Then the mirror map is a function $(C_0 (q), \ldots, C_{m} (q))$ which maps the K\"ahler cone of $X$ to $\cpx^{m+1}$, such that the periods of $\check{X}_{C_0 (q), \ldots, C_{m}(q)}$ coincides with the symplectic areas of two cycles in $X$.

The aim of this section is to give an affirmative answer to this question when $X$ is a toric Calabi-Yau surface:

\begin{theorem} \label{can_coord}
Let $X = X_{\Sigma_m}$ be a toric Calabi-Yau surface, and
$$\check{X} := \left\{(z,u,v) \in \cpx^\times \times \cpx^2 : uv = 1 + \sum_{i=1}^{m} \left( \prod_{j=1}^{i-1} q_j^{i-j} \right) (1+\delta_i) z^i \right\}$$
be the mirror as stated in Theorem \ref{mirror theorem}.  Then the SYZ construction gives a holomorphic volume form $\check{\Omega}$ on $\check{X}$, together with a canonical isomorphism
$$ H_2 (X, \integer) \cong H_2 (\check{X}, \integer) $$
which maps the basis $\{\theta_j := [D_j]\}_{j=1}^{m-1}$ of $H_2(X,\integer)$ to a basis $\{\check{\Theta}_j\}_{j=1}^{m-1}$ of $H_2(\check{X},\integer)$ such that
\begin{equation} \label{eq_can_coord}
-\int_{\theta_j}\omega = \int_{\check{\Theta}_j}\check{\Omega}
\end{equation}
for all $j=1, \ldots, m-1$.
\end{theorem}

Since the mirror map is the SYZ map, we have the expressions
$$ C_i = \left( \prod_{j=1}^{i-1} q_j^{i-j} \right) (1+\delta_i) = \left( \prod_{j=1}^{i-1} q_j^{i-j} \right) \left(\sum_{\alpha} n_{\beta_i + \alpha} q^\alpha \right).$$
Thus the coefficients of the mirror map, when expanded in K\"ahler parameters $q_i$, are open Gromov-Witten invariants.  This gives a geometric understanding of the mirror map.

To prove this theorem, we need to compute the coefficients
$$ \delta_i = \sum_{\alpha \not= 0} n_{\beta_i + \alpha} \exp\left(- \int_\alpha \omega \right) $$
which involve the open Gromov-Witten invariants.  This is done in Section \ref{open_GW}.  Then in Section \ref{proof} we'll prove Theorem \ref{can_coord}.  This includes writing down the holomorphic volume form on $\check{X}$ via SYZ (this is already contained in Section 4.6 of \cite{CLL}), constructing the isomorphism $ H_2 (X, \integer) \cong H_2 (\check{X}, \integer)$, and computing the periods of $\check{X}$.

\subsection{Open Gromov-Witten invariants of toric CY surfaces} \label{open_GW}

In this section we would like to compute the open Gromov-Witten invariants $n_\beta^{\mathrm{T}}$ for a toric fiber $\mathbf{T}$ of a toric Calabi-Yau surface $X_{\Sigma_m}$.  By Proposition \ref{Cho-Oh} it suffices to compute $n_\beta$ for $\beta = \beta_l + \alpha$ where $l \in \{1, \ldots, m-1\}$ and $\alpha \in H_2(X) - \{0\}$.  The result is:

\begin{theorem} \label{Thm_openGW}
Let $X = X_{\Sigma_m}$ be a toric Calabi-Yau surface, $\mathbf{T}$ be a Lagrangian toric fiber, and $\beta = \beta_l + \alpha \in \pi_2(X, \mathbf{T})$ where $\beta_l$ is a basic disk class for $l \in \{1, \ldots, m-1\}$ and $\alpha \in H_2(X)$.  Writing
$$ \alpha = \sum_{k=1}^{m-1} s_k [D_k] $$
where $D_k$ are irreducible compact toric divisors of $X$ and $s_k \in \integer$, then $n_\beta$ equals to $1$ when $\{s_k\}_{k=1}^{m-1}$ is admissible with center $l$, and $0$ otherwise.  A sequence $\{s_k\}_{k=1}^{m-1}$ of integers is said to be admissible with center $l$ if
\begin{enumerate}
\item $s_k \geq 0$ for all $k = 1, \ldots, m-1$.
\item $s_i \leq s_{i+1} \leq s_i + 1$ when $i < l$;
\item $s_i\ge s_{i+1} \ge s_i-1$ when $i \ge l$;
\item $s_{1}, s_{m-1} \leq 1$.
\end{enumerate}
\end{theorem}

As a consequence,
 
\begin{corollary} \label{def_eq}
The defining equation of $\check{X}$ in Theorem \ref{mirror theorem} is
$$uv = (1+z)(1+q_1 z)(1+q_1 q_2 z) \ldots (1 + q_1 \ldots q_{m-1} z) $$
where $$q_j := \exp \left(- \int_{D_j} \omega \right) \textrm{ for } j = 1, \ldots, m-1$$
are the K\"ahler parameters.
\end{corollary}

\begin{proof}
Let $$h(z) = (1+z)(1+q_1 z)(1+q_1 q_2 z) \ldots (1 + q_1 \ldots q_{m-1} z).$$
By direct expansion, the coefficient of $z^p$ (p = 0, \ldots, m) is
$$ \sum_{k_1, \ldots, k_p} \left(\prod_{j=1 \ldots k_1} q_j \right) \ldots \left(\prod_{j=1 \ldots k_p} q_j \right)$$
where the sum is over all $(k_1, \ldots, k_p) \in \integer^p$ such that $0 \leq k_1 < \ldots < k_p \leq m-1$.  Notice that each summand can be written as
$\left(q_1^{p-1} \ldots q_{p-1}\right) q^\alpha $, where
$$\alpha = \left( D_p + \ldots + D_{k_p} \right) + \left( D_{p-1} + \ldots + D_{k_{p-1}} \right) + \left( D_1 + \ldots + D_{k_1} \right). $$
In this form it is clear that $\alpha \in H_2(X)$ is an admissible class with center $p$ in the sense of Theorem \ref{Thm_openGW}.

Conversely, let $\alpha=\sum_{k=1}^{m-1} s_k D_k$ be admissible, and we would like to find $k_j$ such that $\alpha$ is in the above form.  If $\alpha = 0$, we simply set $k_p = p-1, \ldots, k_1 = 0$.  Otherwise let $k_p$ be the greatest integer among $\{1, \ldots, m-1\}$ such that $s_k \not= 0$.  Then by condition (3) of admissibility, $s_j > 0$ for $j = p, \ldots, k_p$.  Thus
$$\alpha = (D_p + \ldots + D_{k_p}) + \sum_{k=1}^{m-1} s'_k D_k $$
with $s'_k \geq 0$.

If $s'_k = 0$ for all $k$, then we are done and set $k_{p-1} = p-2, \ldots, k_1 = 0$.  Otherwise, let $k_{p-1}$ be the greatest integer among $\{1, \ldots, m-1\}$ such that $s'_k \not= 0$.  By condition (3) and (4) of admissibility, $s'_{k_p} = \ldots = s'_{m-1} = 0$ and so $k_{p-1} < k_p$.  Condition (3) implies that $s'_j > 0$ for $j = p, \ldots, k_{p-1}$, and condition (2) implies that $s'_{p-1} > 0$.  Thus we can write
$$\alpha = (D_p + \ldots + D_{k_p}) + (D_{p-1} + \ldots D_{k_{p-1}}) + \sum_{k=1}^{m-1} s''_k D_k.$$

We proceeds by induction, and since $s_1 \leq 1$ by condition (4), it must end with
$$\alpha = \left( D_p + \ldots + D_{k_p} \right) + \left( D_{p-1} + \ldots + D_{k_{p-1}} \right) + \left( D_1 + \ldots + D_{k_1} \right). $$

Now it is clear that the coefficient of $z^p$ is
$$\sum_{\alpha} \left(q_1^{p-1} \ldots q_{p-1}\right) q^\alpha $$
where the summation is over all admissible $\alpha$.  By Theorem \ref{Thm_openGW}, this equals to $$ \left(q_1^{p-1} \ldots q_{p-1}\right) \sum_{\alpha} n_{\beta_p + \alpha} q^\alpha.$$
Thus the defining equation of the mirror can be written as stated.
\end{proof}

Now we prove Theorem \ref{Thm_openGW}.

\begin{proof}[(Proof of Theorem \ref{Thm_openGW})]
It was proved by Chan \cite{Chan10} that for canonical line bundles $X = K_Z$ of toric Fano manifolds $Z$, $n_\beta$ equals to some \emph{closed} Gromov-Witten invariants of the fiberwise compactification $\bar{K}_Z$.  In \cite{LLW10} the arguments are modified slightly to generalize to local Calabi-Yau manifolds $X$.  We now apply them to the current situation that $\dim X = 2$.

To compute $n_{\beta_l + \alpha}$, we consider the toric compactification $Y = \bar{X}$ along the $v_l$ direction: The fan of $\bar{X}$ is convex consisting of rays generated by $v_i = (i,1)$ for $i = 0, \ldots, m$, $(1,0)$, $(-1,0)$ and $v_\infty = -v_l$ (the rays generated by $(1,0)$ and $(-1,0)$ are added to make $\bar{X}$ smooth).  Let $h \in H_2 (\bar{X})$ be the class determined by the intersection properties
$h \cdot D_l = h \cdot D_\infty = 1$
and $ h \cdot D = 0$ for all other irreducible toric divisors $D$ (see Figure \ref{bar_X}).  Intuitively $h$ corresponds to the disk class $\beta_l$.

\begin{figure}[htp]
\caption{A sphere representing $h \in H_2 (\bar{X})$.} \label{bar_X}
\begin{center}
\includegraphics[height=160pt,width=346pt]{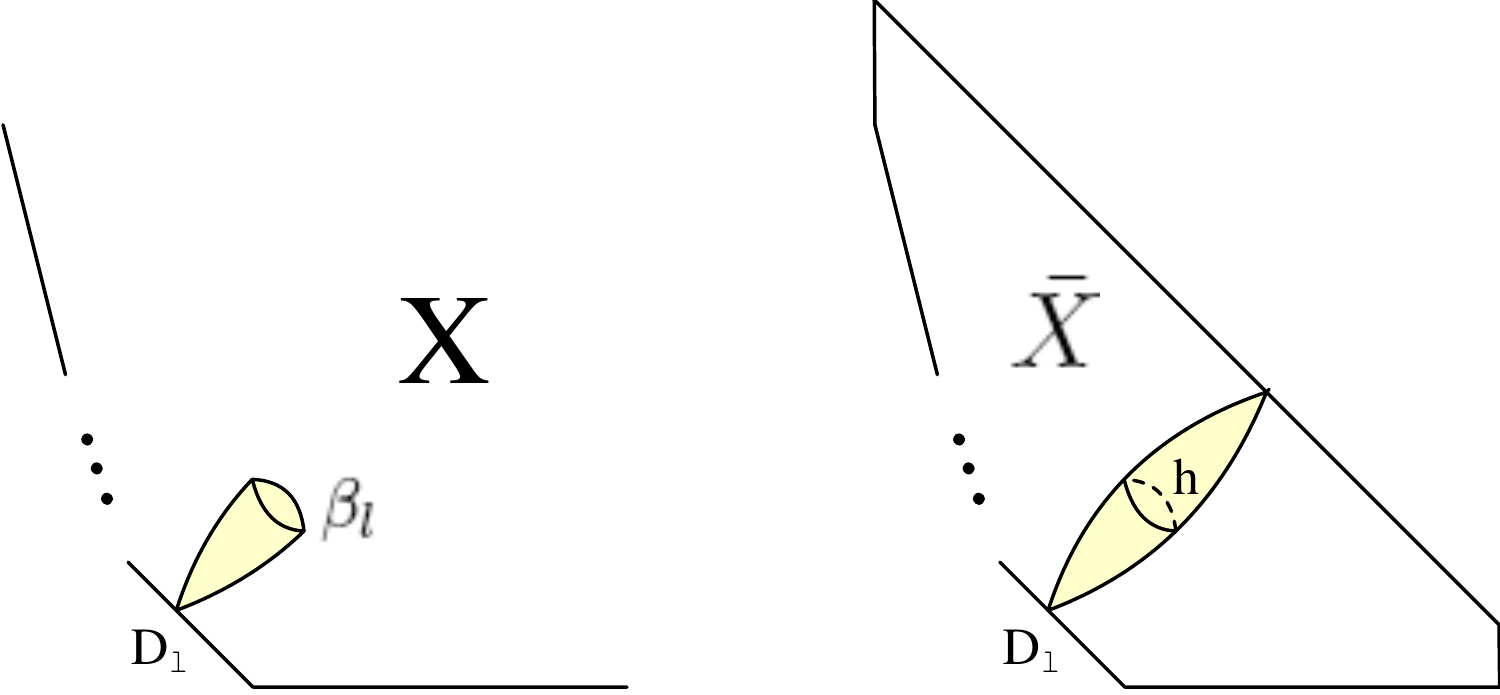}
\end{center}
\end{figure}

By comparing the Kuranishi structures on the open and closed moduli (see Proposition 4.4 in \cite{LLW10} for the details), one has
$$ n_{\beta_l + \alpha} = \mathrm{GW}_{0,1}^{Y, h + \alpha} ([\mathrm{pt}]).$$
The right hand side in the above formula is the genus zero one-pointed closed Gromov-Witten invariant of $Y = \bar{X}$ for the class $h + \alpha$.  Thus it remains to compute $\mathrm{GW}_{0,1}^{Y, h + \alpha} ([\mathrm{pt}])$.

Now we may apply the result by Hu \cite{hu00} and Gathmann \cite{gathmann} which removes the point condition by blow-up:
\begin{equation*}
\mathrm{GW}_{0,1}^{Y, h + \alpha} ([\mathrm{pt}]) = \mathrm{GW}_{0,0}^{\tilde{Y}, \pi^!(h + \alpha) -e}
\end{equation*}
where
$\pi:{\tilde Y} \to Y$ is the blow-up of $Y$ at a point, $e \in H_2 (\tilde Y)$ is the corresponding exceptional class, and $\pi^!(b) := \mathrm{PD}(\pi^*\mathrm{PD}(b))$ for $b \in H_2(\bar{X})$.

Writing $\alpha = \sum_{k=1}^{m-1} s_k [D_k]$, one has
$$\pi^!(h + \alpha) - e = [C] + \sum_{k=1}^{m-1} s_k [D_k]$$
where $C$ is a $(-1)$-curve and $D_k$ are $(-2)$-curves, and their intersection configuration is as shown in Figure \ref{curve_config}.  The Gromov-Witten invariant $\mathrm{GW}_{0,0}^{\tilde{Y}, [C] + \sum_{k=1}^{m-1} s_k [D_k]}$ has already been computed by Bryan-Leung \cite{bryan-leung00}, and the result is that the invariant is $1$ when the sequence $\{s_k\}_{k=1}^{m-1}$ is admissible with center $l$, and $0$ otherwise.  The sense of admissibility for a sequence of integers is the one written in Theorem \ref{Thm_openGW}. 

\begin{figure}[htp]
\caption{A chain of $\proj^1$'s.} \label{curve_config}
\begin{center}
\includegraphics[height=61pt,width=218pt]{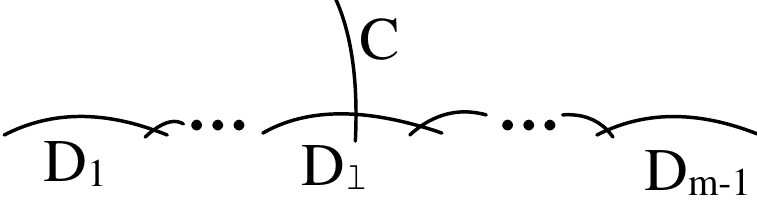}
\end{center}
\end{figure}

\end{proof}

\subsection{Proof of Theorem \ref{can_coord}} \label{proof}
Having an explicit expression of the SYZ mirror (see Corollary \ref{def_eq}), we are prepared to prove the main theorem.

\subsubsection{The holomorphic volume form.} \label{vol_form} First we need to write down the holomorphic volume form on $\check{X}$.  It is known that the semi-flat mirror $\check{X}_0$ has a holomorphic volume form which is simply written as $\der \log z_1 \wedge \der \log z_2$ in terms of the local semi-flat complex coordinates $(z_1, z_2)$.  In \cite{chan08,Chan-Leung}, this (semi-flat) holomorphic volume form is written as Fourier transform of the symplectic form on $X$.  Now recall that $\check{X}$ is glued by two semi-flat pieces 
$\iota_\pm: \check{X}_\pm \to \check{X}$ (Equation \eqref{iota+} and \eqref{iota-}).  One has
$$ \iota^*_\pm (\der \log z \wedge \der \log u) = \der \log z_1 \wedge \der \log z_2 $$
which means that the semi-flat holomorphic volume forms on the two pieces $\check{X}_\pm$ glue up, and it is a direct computation to see that $\der \log z \wedge \der \log u$ extends to give a holomorphic volume form $\check{\Omega}$ on $\check{X}$.  (This has already been discussed in the paper \cite{CLL}).

\subsubsection{$H_2(X) \cong H_2(\check{X})$} \label{S_l}
Now let's turn to the construction of the natural isomorphism $H_2(X) \cong H_2(\check{X})$.  Consider the basis $\{\theta_i = [D_i] \}_{i=1}^{m-1} \subset H_2(X)$.  We would like to perform SYZ transformation on each $D_l$ to give a dual chain $\check{D}_l \subset \check{X}_-$.  We'll see that $\iota_- (\check{D}_l) \subset \check{X}$ is homologous to a chain $C_K$ in $\check{X}$ which limits to a cycle $\check{\Theta}_l$ as $K \to +\infty$.  (Alternatively one may consider the dual chain in the other semi-flat piece $\iota_+: \check{X}_+ \to \check{X}$ instead which leads to the same result).

First we write $D_l$ as
$$\{x \in X: \mu (x) \in [T_l] \times \{0\}; \arg (w(x)-K) = \pi\} $$
where we recall that $T_l$ is the edge of $P$ corresponding to the toric divisor $D_l$, so that $[T_l] \times \{0\}$ is the line segment in $B$ connecting the two points $Q_{l}$ and $Q_{l+1} \in B$ which lie in the discriminant loci of $\mu$ (see Figure \ref{D_l}).  In this expression we can see that $D_l$ is a circle fibration over the line segment $[T_l] \times \{0\}$.

Under T-duality, it induces a dual circle fibration supported in $\check{X}_-$ over the same line segment, which is written explicitly as
$$ \check{D}_l = \{ (z_1,z_2) \in \check{X}_-: \check{\mu} (z_1,z_2) \in [T_l] \times \{0\}; \arg z_1 = \pi \} $$
where we recall that $\check{\mu}$ is the bundle map given in Section \ref{T-duality}.

\begin{figure}[htp]
\caption{A toric divisor.} \label{D_l}
\begin{center}
\includegraphics[height=145pt,width=217pt]{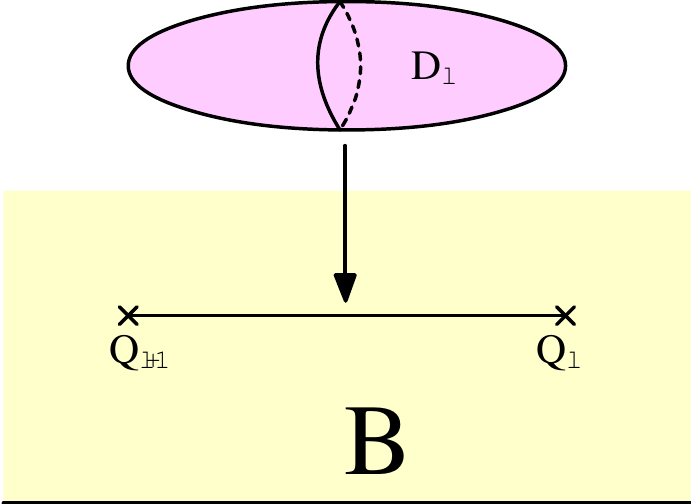}
\end{center}
\end{figure}

From Section \ref{T-duality}, the value of $|z_i|$ (i = 1,2) on the fiber $\check{F}_r$ of $\check{\mu}$ at $r \in U \subset B_0$ is
$ \exp \left( -\int_{\Delta_i (r)} \omega \right) $
where each $\Delta_i (r) \in \pi_2(X, F_r)$ is represented by a disk as shown in Figure \ref{A_n_disks}.  For $r = \check{\mu} (z_1,z_2) \in [T_l] \times \{0\}$, $\prod_{i=1}^{l-1} q_i^{-1} \leq |z_1(r)|
\leq \prod_{i=1}^{l} q_i^{-1}$.  Together with $\arg z_1 = \pi$, one has
$$ z_1 \left( \check{D}_l \right)
= \left[ - \prod_{i=1}^{l} q_i^{-1}, - \prod_{i=1}^{l-1} q_i^{-1} \right].$$

The boundary of $\check{D}_l$ consists of two disjoint circles $C_l$ and $C_{l+1}$ lying in the fibers $\check{F}_{Q_{l}}$ and $\check{F}_{Q_{l+1}}$, on which $z_1$ takes values $-\prod_{i=1}^{l-1} q_i^{-1}$ and $-\prod_{i=1}^{l} q_i^{-1}$ respectively.  Let's denote by $a_j$ the values of $|z_2|^2$ on $\check{F}_{Q_{j}}$, so that the values of $|z_2|^2$ on $C_l$ and $C_{l+1}$ are $a_l$ and $a_{l+1}$ respectively.

Now let's consider the chain 
$$\iota_- (\check{D}_l) \subset \check{X} = \{(u,v,z): uv = g(z)\}$$
(see Figure \ref{dual_chain}).  By Equation \eqref{iota-},
$(z, u, v) = \iota_- (z_1, z_2) = (z_1, z_2 ,z_2^{-1} g (z_1))$
where according to Corollary \ref{def_eq},
$$g(z) = (1+z)(1+q_1 z)(1+q_1 q_2 z) \ldots (1 + q_1 \ldots q_{m-1} z).$$
On the boundaries $\iota_- (C_j)$ ($j = l, l+1$) one has $z = - q_1^{-1} \ldots q_{j-1}^{-1}$, which are roots to the equation $g(z) = 0$, and so
$$z = - \prod_{i=1}^{j-1} q_i^{-1}; v = 0; |u|^2 = a_j.$$
For each $z \in \left[ - \prod_{i=0}^{l} q_i^{-1}, - \prod_{i=0}^{l-1} q_i^{-1} \right]$, the fiber of $\iota_- (\check{D}_l)$ at $z$ is a circle in the cylinder $\{(u,v) \in \cpx^2: uv = g(z)\}$.  Let $f: \left[ - \prod_{i=0}^{l} q_i^{-1}, - \prod_{i=0}^{l-1} q_i^{-1} \right] \to \real$ be an affine linear function which takes values $a_{l+1}$ and $a_{l}$ at the endpoints $- \prod_{i=0}^{l} q_i^{-1}$ and $- \prod_{i=0}^{l-1} q_i^{-1}$ respectively.  Then the fiber $\iota_- (\check{D}_l)|_{z}$ is homotopic to the circle $$\{(u,v) \in \cpx^2: uv = g(z); |u|^2 - |v|^2 = f(z)\}.$$
Thus $\iota_- (\check{D}_l) \subset \check{X}$ is homologous (with boundary being fixed) to the chain
$$\left\{ (u,v,z) \in \check{X}: z \in \left[ - \prod_{i=0}^{l} q_i^{-1}, - \prod_{i=0}^{l-1} q_i^{-1} \right]; |u|^2 - |v|^2 = f(z) \right\}. $$

\begin{figure}[htp]
\caption{The mirror cycles.} \label{dual_chain}
\begin{center}
\includegraphics[height=139pt,width=354pt]{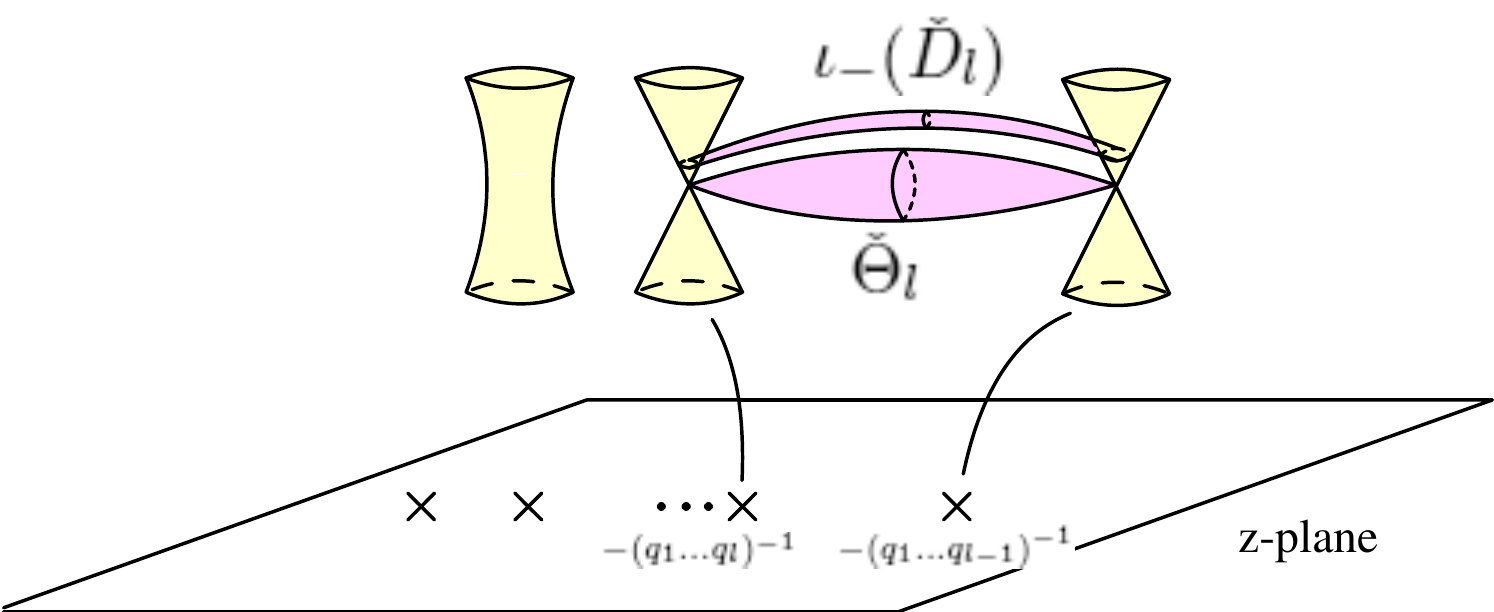}
\end{center}
\end{figure}

Now taking the limit $K \to +\infty$, all $a_j$ tend to $0$, so that $f$ tends to $0$ uniformly.  Thus the above chain limits to
$$ S_l = \left\{ (u, v, z) \in \check{X}: |u| = |v|; z \in \left[ - \prod_{i=0}^{l} q_i^{-1}, - \prod_{i=0}^{l-1} q_i^{-1} \right] \right\} $$
which is a submanifold without boundary in $\check{X}$, and we denote its class by $\check{\Theta}_l \in H_2 (\check{X})$.  $\{\check{\Theta}_l\}_{l=1}^{m-1}$ forms a basis of $H_2 (\check{X})$, and so the map $\theta_l \mapsto \check{\Theta}_l$ gives the required isomorphism $H_2(X) \cong H_2(\check{X})$.

\subsubsection{The periods.} It remains to compute the periods of $\check{X}$ directly:

$$\int_{\check{\Theta}_l} \check{\Omega} = \int_{S_l} \check{\Omega} = \int_{- (q_1 \ldots q_l)^{-1}}^{- (q_1 \ldots q_{l-1})^{-1}} \der \log z = \log q_l = -\int_{\theta_l} \omega.$$

\subsection{Hyper-K\"ahler twist} \label{hyperKaehler}
This subsection aims to relate our SYZ approach with hyper-K\"ahler twist for toric Calabi-Yau surfaces.  Namely, we prove that the hyper-K\"ahler periods of a toric Calabi-Yau surface and its SYZ mirror satisfy Equation \eqref{period map}.

Let's begin with the general theme.  Let $(X, g)$ be an irreducible hyper-K\"ahler manifold, that is, the holonomy group of the Levi-Civita connection induced by the metric $g$ is $\Sp (n)$.  Then $X$ has three parallel orthogonal complex structures $I, J, K$, and all other parallel orthogonal complex structures are given by $a I + b J + c K$ with $|a|^2 + |b|^2 + |c|^2 = 1$, forming an $\sphere{2}$-family.  Moreover we have three parallel K\"ahler forms $\omega_I, \omega_J, \omega_K$ induced from $I, J, K$ respectively.  The holomorphic symplectic form with respect to $I$ is (with a choice of constant multiple) $\Omega_I = - \omega_K + \consti \omega_J$.

Now fixing a basis $\{\theta_i\}_{i=1}^N$ of $H_2(X, \rat)$, we may consider the hyper-K\"ahler periods
\begin{equation*}
\Pi_I = \left( \int_{\theta_1} \omega_I , \ldots ,  \int_{\theta_N} \omega_I \right);
\Pi_J = \left( \int_{\theta_1} \omega_J, \ldots ,  \int_{\theta_N} \omega_J \right);
\Pi_K = \left( \int_{\theta_1} \omega_K, \ldots ,  \int_{\theta_N} \omega_K \right)
\end{equation*}
which span a lightlike subspace in $H^2(X, \real)$.  (When $X$ is compact, this subspace determines the hyper-K\"ahler metric $g$.) 

It is expected that for a hyper-K\"ahler manifold $(X, \omega_I, \omega_J, \omega_K)$, the mirror can be obtained by a hyper-K\"ahler twist.  This means $\check{X}$ is the same as $X$ as a smooth manifold, but with a different choice of complex structure:  $(\check{X}, \check{\omega}_I, \check{\omega}_J, \check{\omega}_K) = (X, \omega_K, \omega_J, \omega_I)$.  In terms of the hyperK\"ahler periods, it means that
\begin{equation} \label{period map}
\check{\Pi}_I = \Pi_K; \check{\Pi}_J = \Pi_J; \check{\Pi}_K = \Pi_K.
\end{equation}
By the identity $\SU (2) = \Sp (1)$, a Calabi-Yau surface is automatically hyper-K\"ahler.  Thus the above expectation about hyper-K\"ahler periods applies to Calabi-Yau surfaces.

\begin{remark}
In general one has to incorporate $B$-fields in mirror symmetry.  Roughly speaking, it means that one has to complexify the K\"ahler cone in order to compare it with the complex moduli of the mirror.  Equation \eqref{period map} is under the condition that we switch off the $B$-field.  When $B$-field is present, the relation between hyper-K\"ahler twist and mirror symmetry is more subtle.  We are thankful to the referee for drawing our attention to this point.
\end{remark}

Now let's come back to our situation that $X = X_{\Sigma_m}$ is a toric Calabi-Yau surface whose toric complex structure is denoted by $I$.  Let $\omega_I$ be the toric symplectic form\footnote{The statement to be made here is in the homology level instead of in the chain level.  In particular we simply use the toric K\"ahler metric instead of the Ricci-flat one, as we only care about its K\"ahler class instead of the actual form.}, and $\Omega = -\omega_K + \consti \omega_J$ be the toric holomorphic volume form.  Via SYZ the mirror $\check{X}$ is constructed (see Theorem \ref{mirror theorem}), which is a complex hypersurface in $\cpx^2 \times \cpx^\times$, so that the standard symplectic form 
$$\der u \wedge \overline{\der u} + \der v \wedge \overline{\der v} + \der \log z \wedge \overline{\der \log z} $$
on $\cpx^2 \times \cpx^\times$ restricts to give a symplectic form $\check{\omega}_I$ on $\check{X}$.  $\check{X}$ is also equipped with a holomorphic volume form $\check{\Omega} = -\check{\omega}_K + \consti \check{\omega}_J$ (Section \ref{vol_form}).  Then as a consequence of Theorem \ref{can_coord}, the mirror $\check{X}$ constructed via SYZ matches with the above discussion:

\begin{corollary}
Let $(X = X_{\Sigma_m}, \omega_I, \omega_J, \omega_K)$ and $(\check{X}, \check{\omega}_I, \check{\omega}_J, \check{\omega}_K)$ be the mirror pairs as discussed above.  Then the corresponding periods $(\Pi_I, \Pi_J, \Pi_K)$ of $X$ and $(\check{\Pi_I}, \check{\Pi_J}, \check{\Pi_K})$ of $\check{X}$ satisfy Equation \ref{period map}.
\end{corollary}
\begin{proof}
Since $D_i$ are complex submanifolds with respect to the toric complex structure, one has
$$\int_{\theta_i} \Omega = -\int_{\theta_i} \omega_K + \consti \int_{\theta_i} \omega_J = 0$$
for all $i$.  On the other hand, $S_l \subset \check{X}$ defined in \ref{S_l} is special Lagrangian with respect to $(\check{\omega}_I, \check{\Omega})$, that is, $\check{\omega}_I|_{S_l} = 0 = \mathrm{Im} \check{\Omega} |_{S_l}$.  Thus
$$\int_{\check{\Theta}_i} \check{\omega}_I = 0 = \int_{\check{\Theta}_i} \check{\omega_J}.$$
This gives $\Pi_J = \check{\Pi}_J = 0$ and $\Pi_K = \check{\Pi}_I = 0$.  From Theorem \ref{can_coord},
$$ \int_{\theta_j} \omega_I = - \int_{\check{\Theta}_j}\check{\Omega} = \int_{\check{\Theta}_j} \check{\omega}_K $$
which means $\Pi_I = \check{\Pi}_K$.
\end{proof}

\bibliographystyle{amsplain}
\bibliography{geometry}

\end{document}